\documentclass[12pt,reqno]{amsart}

\usepackage[cp1251]{inputenc}
\usepackage[T2A]{fontenc}
\usepackage[english]{babel}
\usepackage{amsmath, amsthm, amscd, amsfonts, amssymb, graphicx, color}
\usepackage[bookmarksnumbered, colorlinks, plainpages]{hyperref}

\usepackage{geometry}
\usepackage{cite}
\numberwithin{equation}{section}
\newtheorem{theorem}{Theorem}[section]

\newtheorem{remark}[theorem]{Remark}

\newtheorem{definition}[theorem]{Definition}
\allowdisplaybreaks

\author[Ye. Hnyp, V. Mikhailets, A. Murach]
{Yevheniia Hnyp, Vladimir Mikhailets, Aleksandr Murach}

\title[Parameter-dependent boundary-value problems in Sobolev spaces]
{Parameter-dependent one-dimensional\\ boundary-value problems in Sobolev spaces}

\address{Yevheniia Hnyp \newline
Institute of Mathematics of National Academy of Sciences of Ukraine \\
Tereshchenkivska Str. 3, 01004 Kyiv, Ukraine}
\email{evgeniyagnyp27@gmail.com}

\address{Vladimir Mikhailets \newline
Institute of Mathematics of National Academy of Sciences of Ukraine \\
Tereshchenkivska Str.~3, 01004 Kyiv, Ukraine}
\email{mikhailets@imath.kiev.ua}

\address{Aleksandr Murach \newline
Institute of Mathematics of National Academy of Sciences of Ukraine \\
Tereshchenkivska Str.~3, 01004 Kyiv, Ukraine;
\newline Chernihiv National Pedagogical University, Het'mana Polubotka Str.~53, 14013  Chernihiv, Ukraine}
\email{murach@imath.kiev.ua}



\begin{document}

\maketitle

\begin{abstract}
We consider the most general class of linear boundary-value problems for higher-order ordinary differential systems whose solutions and right-hand sides belong to the corresponding Sobolev spaces. For parameter-depen\-dent problems from this class, we obtain a constructive criterion under which their solutions are continuous in the Sobolev space with respect to the parameter. We also obtain a two-sided estimate for the degree of convergence of these solutions to the solution of the nonperturbed problem. These results are applied to a new broad class of parameter-dependent multipoint boundary-value problems.
 \end{abstract}
\vspace{1cm}

\section{Introduction}\label{sec_Int}

\noindent Questions concerning verification of limit transition in parameter-dependent differential equations arise in various mathematical problems. These questions are best cleared up for ordinary differential systems of the first order. Gikhman~\cite{Gikhman1952}, Krasnoselskii and S.~Krein~\cite{KrasnoselskiiKrein1955}, Kurzweil and Vorel~\cite{KurzweilVorel1957} obtained fundamental results on the continuity in a parameter of solutions to the Cauchy problem for nonlinear differential systems. For linear systems, these results were refined and supplemented by Levin~\cite{Levin1967}, Opial~\cite{Opial1967}, Reid~\cite{Reid1967}, and Nguen~\cite{Nguen1993}.

Parameter-dependent boundary-value problems are less investigated than the Cauchy problem. Kiguradze \cite{Kiguradze1975, Kiguradze1987, Kiguradze2003} and then Ashordia \cite{Ashordia1996} introduced and investigated a class of general linear boundary-value problems for systems of first order differential equations. Kiguradze and Ashordia obtained conditions under which the solutions to the parameter-dependent problems from this class are continuous with respect to the parameter in the normed space $C([a,b],\mathbb{R}^{m})$. Recently  \cite{KodlyukMikhailetsReva2013, MikhailetsChekhanova2015} these results were refined and extended to complex-valued functions and systems of higher order differential equations.

New broad classes of linear boundary-value problems for differential systems are considered in \cite{KodlyukMikhailets2013JMS, GnypKodlyukMikhailets2015UMJ}. These classes relate to the classical scale of complex Sobolev spaces and are introduced for the systems whose right-hand sides and solutions run through the corresponding Sobolev spaces. The boundary conditions for these systems are posed in the most general form by means of an arbitrary continuous linear operator given on the Sobolev space of the solutions. Therefore it is naturally to say that these boundary-value problems are generic with respect to the corresponding Sobolev space.

Generally, the formally adjoint problem and Green formula are not defined for generic problems. Therefore the usual methods of the theory of ordinal differential equations are not applicable to these problems, and their study is of a special interest.

Investigating parameter-dependent generic boundary-value problems, Hnyp, Ko\-dlyuk and Mikhailets \cite{KodlyukMikhailets2013JMS, GnypKodlyukMikhailets2015UMJ} found constructive sufficient conditions under which the solutions to these problems exist and are unique for small values of the positive parameter and are continuous with respect to the parameter in the Sobolev space. (The paper \cite{KodlyukMikhailets2013JMS} deals with the systems of first-order differential equations, whereas \cite{GnypKodlyukMikhailets2015UMJ} investigates the systems of higher-order differential equations.)

The goal of the present paper is to prove that these conditions are necessary as well. So, we prove a constructive criterion for continuity of the solutions with respect to the parameter in Sobolev spaces. Besides, we establish a two-sided estimate for the degree of convergence of the solutions as the parameter approaches zero. An application of this criterion to multipoint parameter-dependent boundary-value problems is also given. Namely, we introduce a new broad class of these problems and obtain explicit sufficient conditions under which the solutions to these problems are continuous with respect to the parameter in Sobolev spaces.

Note that other important classes of generic boundary-value problems are introduced and investigated in papers \cite{MikhailetsChekhanova2014DAN7, Soldatov2015UMJ}. These classes relate to the classical scale of normed spaces of continuously differential functions. Sufficient conditions for the continuous dependence on the parameter of solutions to these problems are established in these papers.

The above-mentioned results were applied to investigation of multipoint boun\-da\-ry-value problems \cite{Kodlyuk2012Dop11} and Green matrixes of boundary-value problems \cite{KodlyukMikhailetsReva2013, MikhailetsChekhanova2015}, to the spectral theory of differential operators with singular coefficients \cite{GoriunovMikhailets2010MN, GoriunovMikhailetsPankrashkin2013EJDE, GoriunovMikhailets2012UMJ}. The latter application stimulates us to consider differential equations with complex-valued coefficients and right-hand sides.

The approach used in the present paper can be applied to investigation of boundary-value problems which are generic with respect to other normed function spaces (see \cite{MikhailetsMurachSoldatov2016EJQTDE, MikhailetsMurachSoldatov2016MFAT}).

\section{Statement of problem and main results}\label{sec2}

\noindent We arbitrarily choose integers $n\geq0$ and $m,r\geq1$, a real number $p\in[1,\infty)$ and compact interval $[a,b]\subset\nobreak\mathbb{R}$.
We use the complex Sobolev spaces
\begin{equation*}
W_p^n:=W^n_p([a,b],\mathbb{C}),\quad
(W^n_p)^m:=W_p^n([a,b],\mathbb{C}^m),
\end{equation*}
and
\begin{equation*}
(W_p^n)^{m\times m}:=W_p^n([a,b],\mathbb{C}^{m\times m})
\end{equation*}
formed respectively by scalar, vector-valued, and matrix-valued functions defined on $[a,b]$. Recall that the norm in the Banach space $W_p^n$ is defined by the formula
\begin{equation*}
\|x\|_{n,p}:=\Biggl(\,\sum\limits_{j=0}^{n}\,\int\limits_a^b
|x^{(j)}(t)|^p\,dt\Biggr)^{1/p}\quad\mbox{for}\quad x\in W_p^n.
\end{equation*}
Note that $W_p^0$ is the Lebesgue space $L_p:=L_p([a,b],\mathbb{C})$. The norms in the Banach spaces
$\left(W^n_p\right)^m$ and $\left(W_p^n\right)^{m\times m}$ are the sums of the norms in $W_p^n$ of all components of vector-valued or matrix-valued functions from these spaces. We denote these norms by $\|x\|_{n,p}$ as well; it will be always clear from context in which Sobolev space (scalar or vector-valued or matrix-valued functions) these norms are considered.

Let $\varepsilon_0>0$, and let the parameter $\varepsilon$ run through $[0,\varepsilon_0)$. We consider the following parameter-dependent boundary-value problem for a system of $m$ linear differential equations of order~$r$:
\begin{gather}\label{bound_pr_eps_1}
L(\varepsilon)y(t,\varepsilon)\equiv y^{(r)}(t,\varepsilon)+ \sum\limits_{j=1}^{r}A_{r-j}(t,\varepsilon)y^{(r-j)}(t,\varepsilon)=
f(t,\varepsilon),\quad a\leq t\leq b,\\
B(\varepsilon)y(\cdot,\varepsilon)=c(\varepsilon). \label{bound_pr_eps_2}
\end{gather}
Here, for every $\varepsilon\in[0,\varepsilon_0)$, the unknown vector-valued function $y(\cdot,\varepsilon)$ belongs to the space $(W^{n+r}_p)^m$, and we arbitrarily choose the matrix-valued functions $A_{r-j}(\cdot,\varepsilon)\in(W_p^n)^{m\times m}$ with $j\in\{1,\ldots,r\}$, vector-valued function
$f(\cdot,\varepsilon)\in(W^n_p)^m$, vector $c(\varepsilon)\in\mathbb{C}^{rm}$, and
continuous linear operator
\begin{equation}\label{oper_B_class}
B(\varepsilon):(W^{n+r}_p)^m\rightarrow\mathbb{C}^{rm}.
\end{equation}
Throughout the paper, we interpret vectors as columns. Note that the functions $A_{r-j}(t,\varepsilon)$ are not assumed to have any regularity with respect to $\varepsilon$.

Let us indicate the sense in which equation \eqref{bound_pr_eps_1} is understood. If $n\geq1$, then the solution $y(\cdot,\varepsilon)\in(W^{n+r}_p)^m$ belongs to the space $(C^{r})^{m}:=C^{r}([a,b],\mathbb{C}^{m})$ by the Sobolev embedding theorem and then equality \eqref{bound_pr_eps_1} should be fulfilled at every point $t\in[a,b]$. (In this case, all components of $A_{r-j}(\cdot,\varepsilon)$ and $f(\cdot,\varepsilon)$ are continuous on $[a,b]$.) If $n=0$, then $y(\cdot,\varepsilon)\in(W^{n+r}_p)^m\subset(C^{r-1})^{m}$ by this theorem and, moreover, $y^{(r-1)}(\cdot,\varepsilon)$ is absolutely continuous on $[a,b]$. Hence in this case, the classical derivative $y^{(r)}(\cdot,\varepsilon)$ exists almost everywhere on $[a,b]$, and therefore we require equality \eqref{bound_pr_eps_1} to be fulfilled  almost everywhere on $[a,b]$.

Note that the boundary condition \eqref{bound_pr_eps_2} with the arbitrary continuous operator $B(\varepsilon)$ is the most general for the differential system \eqref{bound_pr_eps_1}. Indeed, if the right-hand side $f(\cdot,\varepsilon)$ of the system runs through the whole space  $(W^{n}_p)^m$, then the solution $y(\cdot,\varepsilon)$ to the system runs through the whole space $(W^{n+r}_p)^m$. This condition covers both all kinds of classical boundary conditions (such as initial conditions of the Cauchy problem, multipoint and integral boundary conditions) and nonclassical boundary conditions which contain the derivatives $y^{(k)}(\cdot,\varepsilon)$, with $r\leq k\leq n+r$, of the unknown function. Therefore the boundary-value problem \eqref{bound_pr_eps_1}, \eqref{bound_pr_eps_2} is called generic with respect to the Sobolev space $W^{n+r}_p$.

For every $\varepsilon\in[0,\varepsilon_{0})$, the continuous linear operator \eqref{oper_B_class} admits the following unique representation \cite[Section~0.1]{IoffeTixonov1974}:
\begin{equation}\label{oper_B_represent}
B(\varepsilon)y=\sum\limits_{k=1}^{n+r}\alpha_k(\varepsilon) y^{(k-1)}(a)+\int\limits_a^b\Phi(t,\varepsilon)y^{(n+r)}(t)dt
\end{equation}
for arbitrary $y\in(W^{n+r}_p)^{m}$. Here, each $\alpha_k(\varepsilon)$ is a number $rm\times m$-matrix, and $\Phi(\cdot,\varepsilon)$ is a matrix-function from the space $(L_q)^{rm\times m}$ with the index $q\in(1,\infty]$ subject to $1/p+1/q=1$.

With the boundary-value problem \eqref{bound_pr_eps_1}, \eqref{bound_pr_eps_2}, we associate the continuous linear operator
\begin{equation}\label{6.LBe}
(L(\varepsilon),B(\varepsilon)):(W^{n+r}_p)^{m}\rightarrow (W^{n}_p)^{m}\times\mathbb{C}^{rm}.
\end{equation}
According to \cite[Theorem~1]{GnypKodlyukMikhailets2015UMJ}, operator \eqref{6.LBe} is Fredholm with zero index for every $\varepsilon\in[0,\varepsilon_{0})$.

Let us now give our basic

\begin{definition}\label{defin_2} \rm
We say that the solution to the boundary-value problem  \eqref{bound_pr_eps_1}, \eqref{bound_pr_eps_2} depends continuously on the parameter $\varepsilon$ at $\varepsilon=0$ if the following two conditions are fulfilled:
\begin{itemize}
\item [$(\ast)$] There exists a positive number $\varepsilon_{1}<\varepsilon_{0}$ such that for arbitrary   $\varepsilon\in\nobreak[0,\varepsilon_{1})$, $f(\cdot,\varepsilon)\in(W_p^n)^m$, and $c(\varepsilon)\in\mathbb{C}^{rm}$ this problem has a unique solution $y(\cdot,\varepsilon)\in(W_p^{n+r})^{m}$.
\item [$(\ast\ast)$] The convergence of the right-hand sides
$f(\cdot,\varepsilon)\to f(\cdot,0)$ in $(W_p^n)^{m}$ and $c(\varepsilon)\to c(0)$ in $\mathbb{C}^{rm}$ as $\varepsilon\to0+$
implies the convergence of the solutions $y(\cdot,\varepsilon)\to y(\cdot,0)$ in $(W_p^{n+r})^{m}$ as $\varepsilon\to0+$.
\end{itemize}
\end{definition}

\begin{remark}\rm
We will obtain an equivalent of Definition~\ref{defin_2} if we replace $(\ast\ast)$ with the following condition: $y(\cdot,\varepsilon)\to y(\cdot,0)$ in $(W_p^{n+r})^{m}$ as $\varepsilon\to0+$ provided that $f(\cdot,\varepsilon)=f(\cdot,0)$ and $c(\varepsilon)=c(0)$ for all sufficiently small $\varepsilon>0$. Indeed, this condition together with  $(\ast)$ means that the operator $(L(\varepsilon),B(\varepsilon))^{-1}$, inverse of \eqref{6.LBe}, converges strongly to $(L(0),B(0))^{-1}$ as $\varepsilon\to\infty$. The latter property implies $(\ast\ast)$.
\end{remark}

Following \cite[Section~3]{GnypKodlyukMikhailets2015UMJ}, we consider the next two conditions on the left-hand sides of this problem.

\medskip

\noindent\textbf{Limit Conditions} as $\varepsilon\to0+$:
\begin{itemize}
  \item  [(I)] $A_{r-j}(\cdot,\varepsilon)\to A_{r-j}(\cdot,0)$ in $(W_p^n)^{m\times m}$ for each $j\in\{1\ldots r\}$;
  \item [(II)] $B(\varepsilon)y\to B(0)y$ in $\mathbb{C}^{rm}$ for every $y\in(W_p^{n+r})^m$.
\end{itemize}

We also consider

\medskip

\noindent\textbf{Condition (0).} \it  The homogeneous limiting boundary-value problem
\begin{equation*}
L(0)y(t,0)=0,\quad a\leq t\leq b,\qquad B(0)y(\cdot,0)=0
\end{equation*}
has only the trivial solution.\rm

\medskip

Our main result is

\begin{theorem}\label{main_th}
The solution to the boundary-value problem \eqref{bound_pr_eps_1}, \eqref{bound_pr_eps_2} depends continuously on the parameter $\varepsilon$ at $\varepsilon=0$ if and only if this problem satisfies Condition~{\rm(0)} and Limit Conditions {\rm(I)} and {\rm(II)}.
\end{theorem}

We supplement this theorem with a two-sided estimate of the deviation $\|y(\cdot,0)-y(\cdot,\varepsilon)\|_{n+r,p}$ for sufficiently small $\varepsilon>0$. Let
\begin{equation*}
d_{n,p}(\varepsilon):=
\|L(\varepsilon)y(\cdot,0)-f(\cdot,\varepsilon)\|_{n,p}+
\|B(\varepsilon)y(\cdot,0)-c(\varepsilon)\|_{\mathbb{C}^{rm}}.
\end{equation*}

\begin{theorem}\label{4.th-bound_2}
Assume that the boundary-value problem \eqref{bound_pr_eps_1}, \eqref{bound_pr_eps_2} satisfies Condition~{\rm(0)} and Limit Conditions {\rm(I)} and {\rm(II)}. Then there exist positive numbers $\varepsilon_{2}<\varepsilon_{1}$, $\gamma_{1}$, and $\gamma_{2}$ such that for every   $\varepsilon\in(0,\varepsilon_{2})$ we have the two-sided estimate
\begin{equation}\label{6.bound}
\gamma_{1}\,d_{n,p}(\varepsilon)\leq
\|y(\cdot,0)-y(\cdot,\varepsilon)\|_{n+r,p}
\leq\gamma_{2}\,d_{n,p}(\varepsilon).
\end{equation}
Here, the numbers $\varepsilon_{2}$, $\gamma_{1}$, and $\gamma_{2}$ do not depend on $y(\cdot,0)$ and $y(\cdot,\varepsilon)$.
\end{theorem}

In this theorem, we can interpret $\|y(\cdot,0)-y(\cdot,\varepsilon)\|_{n+r,p}$ and $d_{n,p}(\varepsilon)$ respectively as an error and discrepancy of the solution $y(\cdot,\varepsilon)$ to the problem \eqref{bound_pr_eps_1}, \eqref{bound_pr_eps_2} provided that $y(\cdot,0)$ is considered as an approximate solution to this problem. In this sense, formula \eqref{6.bound} means that the error and discrepancy are of the same degree as $\varepsilon\to0+$.

Note that Theorems \ref{main_th} and \ref{4.th-bound_2} are new even for classical boundary conditions, in which the orders of derivatives of the unknown function are less then~$r$. We will prove these theorems in Section~\ref{sec3}.

\begin{remark}\label{rem2.4}\rm
It follows directly from representation \eqref{oper_B_represent} and the criterion for weak convergence of continuous linear operators on $L_p$ \cite[Chapter~VIII, Section~3.3]{KantorovichAkilov64} that Limit Condition~(II) is equivalent to the following system of conditions as $\varepsilon\to0+$:
\begin{enumerate}
\item[(2a)] $\alpha_k(\varepsilon)\to\alpha_k(0)$ for each $k\in\{1,\dots,n+r\}$; \smallskip
\item[(2b)] $\|\Phi(\cdot,\varepsilon)\|_{(L_{q})^{rm\times m}}=O(1)$; \smallskip
\item[(2c)] $\int\limits_a^t\Phi(s,\varepsilon)ds\to\int\limits_a^t\Phi(s,0)ds$ for every $t\in(a,b]$.
\end{enumerate}
It is useful to compare these conditions with the criterion for the convergence $B(\varepsilon)\to B(0)$ in the uniform operator topology as $\varepsilon\to0+$. This convergence is equivalent to the system of conditions (2a) and $\Phi(\cdot,\varepsilon)\to\Phi(\cdot,0)$ in $(L_{q})^{rm\times m}$ as $\varepsilon\to0+$. The latter condition is evidently stronger than the pair of conditions (2b) and (2c). Note also that Limit Condition~II means the convergence $B(\varepsilon)\to B(0)$ as $\varepsilon\to0+$ in the strong operator topology and that this topology is not metrizable because it does not satisfy the first separability axiom (see, e.g., \cite[Chapter~VI, Section~1]{ReedSimon80}).
\end{remark}

\section{Proofs of the main results}\label{sec3}

\begin{proof}[Proof of Theorem $\ref{main_th}$]
The sufficiency of the system of Condition (0) and Limit Condition (I) and (II) for the problem \eqref{bound_pr_eps_1}, \eqref{bound_pr_eps_2} to satisfy Definition~\ref{defin_2} is proved in \cite[Theorem~1.1]{KodlyukMikhailets2013JMS} for $r=1$ and in \cite[Theorem~3]{GnypKodlyukMikhailets2015UMJ} for $r\geq2$.  Thus, we should establish necessity only.

Assume that this problem satisfies Definition~\ref{defin_2}. Then, of course, Condition (0) is fulfilled. It remains to prove that the problem satisfies both Limit Condition (I) and~(II). We divide our reasoning into three steps.

\emph{Step $1$.} Here, we will prove that the boundary-value problem  \eqref{bound_pr_eps_1}, \eqref{bound_pr_eps_2} satisfies Limit Condition~(I).

If $r\geq2$, we will previously reduce the problem \eqref{bound_pr_eps_1}, \eqref{bound_pr_eps_2} to a boundary-value problem for system of first-order differential equations. Let the parameter $\varepsilon\in[0,\varepsilon_{0})$. As usual, we put
\begin{gather}\label{function-x}
x(\cdot,\varepsilon):=
\mathrm{col}(y(\cdot,\varepsilon),y'(\cdot, \varepsilon),\ldots,y^{(r-1)}(\cdot,\varepsilon))\in(W_p^{n+1})^{rm},\\
\widetilde{f}(\cdot,\varepsilon):=
\mathrm{col}\bigl(0,\ldots,0,f(\cdot,\varepsilon))
\in(W_p^{n})^{rm},\notag
\end{gather}
and
\begin{equation*}
\widetilde{A}(\cdot,\varepsilon):=\left(
\begin{array}{ccccc}
O_m & I_m & O_m & \ldots & O_m \\
O_m & O_m & I_m & \ldots & O_m \\
\vdots & \vdots & \vdots & \ddots & \vdots \\
O_m & O_m & O_m & \ldots & I_m \\
A_0(\cdot,\varepsilon) & A_1(\cdot,\varepsilon) & A_2(\cdot,\varepsilon) & \ldots & A_{r-1}(\cdot,\varepsilon)\\
\end{array}\right)\in(W_p^{n})^{rm\times rm}.
\end{equation*}
Here, $O_m$ and $I_m$ stand respectively for zero and identity  $(m \times m)$-matrix. Bearing representation \eqref{oper_B_represent} in mind, we also set
\begin{equation}\label{6.Be_fo}
\widetilde{B}(\varepsilon)x:=
\sum\limits_{k=1}^{r-1}\alpha_{k}(\varepsilon)x_{k}(a)+
\sum\limits_{k=r}^{n+r}\alpha_{k}(\varepsilon)x_{r}^{(k-r)}(a)
+\int\limits_{a}^{b}\Phi(t,\varepsilon)x_{r}^{(n+1)}(t)dt
\end{equation}
for every vector-valued function $x=\mathrm{col}(x_{1},\ldots,x_{r})$ with $x_{1},\ldots,x_{r}\in(W_p^{n+1})^{m}$. The linear mapping $x\mapsto\widetilde{B}(\varepsilon)x$ acts continuously from $(W_p^{n+1})^{rm}$ to $\mathbb{C}^{rm}$. Let us consider the boundary-value problem
\begin{gather}\label{6.fo}
x'(t,\varepsilon)+\widetilde{A}(t,\varepsilon)x(t,\varepsilon)=
\widetilde{f}(t,\varepsilon),\quad a\leq t\leq b,\\
\widetilde{B}(\varepsilon)x(\cdot,\varepsilon)=c(\varepsilon). \label{6.BC_fo}
\end{gather}
It is evident that a function $y(\cdot,\varepsilon)\in(W^{n+r}_p)^m$ is a solution to the problem \eqref{bound_pr_eps_1}, \eqref{bound_pr_eps_2} if and only if the function \eqref{function-x} is a solution to the problem \eqref{6.fo}, \eqref{6.BC_fo}. In the $r=1$ case, we put $x(\cdot,\varepsilon):=y(\cdot,\varepsilon)$, $\widetilde{f}(\cdot,\varepsilon):=f(\cdot,\varepsilon)$,  $\widetilde{A}(\cdot,\varepsilon):=A_{0}(\cdot,\varepsilon)$, and $\widetilde{B}(\varepsilon):=B(\varepsilon)$ for the sake of uniformity in notation on Step~1, then the problem \eqref{bound_pr_eps_1}, \eqref{bound_pr_eps_2} coincides with the problem \eqref{6.fo}, \eqref{6.BC_fo}.

Limit Condition~(I) is equivalent to the convergence  $\widetilde{A}(\cdot,\varepsilon)\to\widetilde{A}(\cdot,0)$ in the space $(W_p^{n})^{rm\times rm}$ as $\varepsilon\to0+$. Let us prove this convergence.

To this end we note the following: if $\widetilde{f}(\cdot,\varepsilon)$ and $c(\varepsilon)$ do not depend on $\varepsilon\in[0,\varepsilon_{1})$, then $y(\cdot,\varepsilon)\to y(\cdot,0)$ in $(W^{n+r}_p)^{m}$ as $\varepsilon\to0+$ by condition~($\ast\ast$) of Definition~\ref{defin_2}. The latter convergence is equivalent to that $x(\cdot,\varepsilon)\to x(\cdot,0)$ in $(W^{n+1}_p)^{rm}$ as $\varepsilon\to0+$.

Given $\varepsilon\in[0,\varepsilon_{1})$, we consider the matrix boundary-value problem
\begin{gather}\label{matrix-eq}
X'(t,\varepsilon)+\widetilde{A}(t,\varepsilon)X(t,\varepsilon)=O_{rm},
\quad a\leq t\leq b,\\
[\widetilde{B}(\varepsilon)X(\cdot,\varepsilon)]=I_{rm}.
\label{matrix-bound-cond}
\end{gather}
Here, $X(\cdot,\varepsilon):=(x_{j,k}(\cdot,\varepsilon))_{j,k=1}^{rm}$ is an unknown matrix-valued function from the space $(W^{n+1}_p)^{rm\times rm}$, and
$$
[\widetilde{B}(\varepsilon)X(\cdot,\varepsilon)]:=
\left(\widetilde{B}(\varepsilon)
\begin{pmatrix}
x_{1,1}(\cdot,\varepsilon)\\
\vdots \\
x_{rm,1}(\cdot,\varepsilon)\\
\end{pmatrix}
\;\ldots\;
\widetilde{B}(\varepsilon)\begin{pmatrix}
x_{1,rm}(\cdot,\varepsilon)\\
\vdots \\
x_{rm,rm}(\cdot,\varepsilon)\\
\end{pmatrix}\right).
$$

The problem \eqref{matrix-eq}, \eqref{matrix-bound-cond} is a union of $rm$ boundary-value problems \eqref{6.fo}, \eqref{6.BC_fo} whose right-hand sides do not depend on~$\varepsilon$. Therefore it follows directly from our assumption that this problem has a unique solution $X(\cdot,\varepsilon)\in(W^{n+1}_p)^{rm\times rm}$, and, moreover, $X(\cdot,\varepsilon)\to X(\cdot,0)$ in the space $(W^{n+1}_p)^{rm\times rm}$ as $\varepsilon\to0+$. Note that $\det X(t,\varepsilon)\neq0$ for every $t\in[a,b]$; otherwise the function columns of $X(\cdot,\varepsilon)$ would be linear dependent, contrary to~\eqref{matrix-bound-cond}. Since $(W^{n+1}_p)^{rm\times rm}$ is a Banach algebra, the latter convergence implies that $(X(\cdot,\varepsilon))^{-1}\to(X(\cdot,0))^{-1}$ in $(W^{n+1}_p)^{rm\times rm}$ as $\varepsilon\to0+$. Besides, $X'(\cdot,\varepsilon)\to X'(\cdot,0)$ in $(W^{n}_p)^{rm\times rm}$ as $\varepsilon\to0+$. Hence, owing to \eqref{matrix-eq}, we obtain the convergence
\begin{equation*}
\widetilde{A}(\cdot,\varepsilon)=
-X'(\cdot,\varepsilon)(X(\cdot,\varepsilon))^{-1} \rightarrow-X'(\cdot,0)(X(\cdot,0))^{-1}=\widetilde{A}(\cdot,0)
\end{equation*}
in the space $(W^{n}_p)^{rm\times rm}$ as $\varepsilon\to0+$. Here, we use the fact that $(W^{n}_p)^{rm\times rm}$ is a Banach algebra if $n\geq1$, and, besides, that $(X(\cdot,\varepsilon))^{-1}\to(X(\cdot,0))^{-1}$
in $C([a,b],\mathbb{C}^{rm\times rm})$ when we reason in the $n=0$ case. Thus, the problem  \eqref{bound_pr_eps_1}, \eqref{bound_pr_eps_2} satisfies Limit Condition~(I). Specifically,
\begin{equation}\label{bound-norm-A}
\|A_{r-j}(\varepsilon)\|_{n,p}=O(1)\quad\mbox{as}\quad\varepsilon\to0+
\quad\mbox{for each}\quad j\in\{1,\ldots,r\}.
\end{equation}

\emph{Step~$2$.} Let us prove that
\begin{equation}\label{bound-norm-B}
\|B(\varepsilon)\|=O(1)\quad\mbox{as}\quad\varepsilon\to0+;
\end{equation}
here, $\|B(\varepsilon)\|$ denotes the norm of the continuous operator  \eqref{oper_B_class}. Suppose the contrary, i.e. there exists a number sequence $(\varepsilon^{(k)})_{k=1}^{\infty}\subset(0,\varepsilon_{1})$ such that $\varepsilon^{(k)}\to0$ and
\begin{equation}\label{norm-B-infty}
0<\|B(\varepsilon^{(k)})\|\to\infty\quad\mbox{as}\quad k\to\infty.
\end{equation}
For every integer $k\geq1$, we choose a function $w_{k}\in(W^{n+r}_p)^{m}$ such that
\begin{equation}\label{norm-B-1/2}
\|w_{k}\|_{n+r,p}=1\quad\mbox{and}\quad \|B(\varepsilon^{(k)})w_{k}\|_{\mathbb{C}^{rm}}\geq
\frac{1}{2}\,\|B(\varepsilon^{(k)})\|.
\end{equation}
We let
$$
y(\cdot,\varepsilon^{(k)}):=\|B(\varepsilon^{(k)})\|^{-1}\,w_{k},\quad
f(\cdot,\varepsilon^{(k)}):=
L(\varepsilon^{(k)})\,y(\cdot,\varepsilon^{(k)}),
$$
and
$$
c(\varepsilon^{(k)}):=B(\varepsilon^{(k)})\,y(\cdot,\varepsilon^{(k)}).
$$
It follows from \eqref{norm-B-infty} and \eqref{norm-B-1/2} that
\begin{equation}\label{convergence-y}
y(\cdot,\varepsilon^{(k)})\to0\quad\mbox{in}\quad(W^{n+r}_p)^{m}
\quad\mbox{as}\quad k\to\infty.
\end{equation}
Hence,
\begin{equation}\label{convergence-f}
f(\cdot,\varepsilon^{(k)})\to0\quad\mbox{in}\quad(W^{n}_p)^{m}
\quad\mbox{as}\quad k\to\infty
\end{equation}
because the problem \eqref{bound_pr_eps_1}, \eqref{bound_pr_eps_2} satisfies Limit Condition~(I) according to Step~1. Besides, it follows directly from \eqref{norm-B-1/2} that $1/2\leq\|c(\varepsilon^{(k)})\|_{\mathbb{C}^{rm}}\leq1$. Therefore, passing to a subsequence of $(\varepsilon^{(k)})_{k=1}^{\infty}$, we can assume that
\begin{equation}\label{convergence-c}
c(\varepsilon^{(k)})\to c(0)\quad\mbox{in}\quad\mathbb{C}^{rm}\quad
\mbox{as}\quad k\to\infty\quad\mbox{for some}\quad c(0)\neq0.
\end{equation}

Recall that, for every integer $k\geq1$, the vector-valued function $y(\cdot,\varepsilon^{(k)})\in(W^{n+r}_p)^{m}$ is a unique solution to the boundary-value problem
\begin{gather*}
L(\varepsilon^{(k)})y(t,\varepsilon^{(k)})=f(t,\varepsilon^{(k)}),  \quad a\leq t\leq b,\\
B(\varepsilon^{(k)})y(\cdot,\varepsilon^{(k)})=c(\varepsilon^{(k)}).
\end{gather*}
Since this problem satisfies condition $(\ast\ast)$ of Definition~\ref{defin_2}, it follows from \eqref{convergence-f} and
\eqref{convergence-c} that the function $y(\cdot,\varepsilon^{(k)})$ converges in $(W^{n+r}_p)^{m}$ to the unique solution $y(\cdot,0)$ of the boundary-value problem
\begin{gather}\notag
L(0)y(t,0)=0,\quad a\leq t\leq b,\\
B(0)y(\cdot,0)=c(0).\label{bound-cond-0}
\end{gather}
But $y(\cdot,0)=0$ by \eqref{convergence-y}, which contradicts \eqref{bound-cond-0} in view of $c(0)\neq0$. Hence, the assumption made at the beginning of Step~2 is wrong; so, we have proved \eqref{bound-norm-B}.

\emph{Step $3$.} Let us now prove that the boundary-value problem  \eqref{bound_pr_eps_1}, \eqref{bound_pr_eps_2} satisfies Limit Condition~(II). According to \eqref{bound-norm-A} and
\eqref{bound-norm-B}, there exist numbers $\gamma'>0$ and  $\varepsilon'\in(0,\varepsilon_{1})$ such that
\begin{equation}\label{bound-norm-LB}
\|(L(\varepsilon),B(\varepsilon))\|\leq\gamma'
\quad\mbox{for every}\quad\varepsilon\in[0,\varepsilon').
\end{equation}
Here, $\|(L(\varepsilon),B(\varepsilon))\|$ denotes the norm of the continuous operator \eqref{6.LBe}. We arbitrarily choose a function
$y\in(W^{n+r}_p)^{m}$ and put $f(\cdot,\varepsilon):=L(\varepsilon)y$ and $c(\varepsilon):=B(\varepsilon)y$ for every   $\varepsilon\in[0,\varepsilon_{0})$. Then,
\begin{equation}\label{y-L(e)B(e)}
y=(L(\varepsilon),B(\varepsilon))^{-1}
(f(\cdot,\varepsilon),c(\varepsilon))\quad\mbox{for every}\quad
\varepsilon\in[0,\varepsilon').
\end{equation}
Here, of course, $(L(\varepsilon),B(\varepsilon))^{-1}$ denotes the inverse operator to \eqref{6.LBe}. (Recall that the operator \eqref{6.LBe} is invertible by condition ($\ast$) of Definition~\eqref{defin_2}.) Using \eqref{bound-norm-LB} and \eqref{y-L(e)B(e)}, we obtain the following relations as $\varepsilon\to0+$:
\begin{align*}
&\bigl\|B(\varepsilon)y-B(0)y\bigr\|_{\mathbb{C}^{rm}}\leq
\bigl\|(f(\cdot,\varepsilon),c(\varepsilon))-
(f(\cdot,0),c(0))\bigr\|_{(W^{n}_p)^{m}\times\mathbb{C}^{rm}}\\
&\leq\gamma'\,\bigl\|(L(\varepsilon),B(\varepsilon))^{-1}
\bigl((f(\cdot,\varepsilon),c(\varepsilon))-
(f(\cdot,0),c(0))\bigr)\bigr\|_{n+r,p}\\
&=\gamma'\,\bigl\|(L(0),B(0))^{-1}(f(\cdot,0),c(0))-
(L(\varepsilon),B(\varepsilon))^{-1}(f(\cdot,0),c(0))\bigr\|_{n+r,p}
\to0.
\end{align*}
The latter convergence is due to condition $(\ast\ast)$ of Definition~\ref{defin_2}. Thus, since the function $y\in(W^{n+r}_p)^{m}$ is arbitrary, we have proved that Limit Condition~(II) is satisfied.
\end{proof}

\begin{proof}[Proof of Theorem $\ref{4.th-bound_2}$]
Let us first prove the left-hand side of~\eqref{6.bound}. Limit Conditions (I) and (II) imply the strong convergence $(L(\varepsilon),B(\varepsilon))\to(L(0),B(0))$ as $\varepsilon\to0+$ of the continuous operators from $(W^{n+r}_p)^{m}$ to $(W^{n}_p)^{m}\times\mathbb{C}^{rm}$. Hence, there exist numbers $\gamma'>0$ and $\varepsilon'\in(0,\varepsilon_{0})$ that the norm of the operator $(L(\varepsilon),B(\varepsilon))$ satisfies condition \eqref{bound-norm-LB}. Indeed, if this condition were not fulfilled, there would exist a sequence of positive numbers $(\varepsilon^{(k)})_{k=1}^{\infty}$ such that $\varepsilon^{(k)}\to0$ and $\|(L(\varepsilon^{(k)}),B(\varepsilon^{(k)})\|\to\infty$ as $k\to\infty$, which would contradict the above-mentioned strong convergence in view of Banach-Steinhaus Theorem. Now, owing to
\eqref{bound-norm-LB}, we conclude that
\begin{align*}
d_{n,p}(\varepsilon)&=
\|(L(\varepsilon),B(\varepsilon))(y(\cdot,0)-y(\cdot,\varepsilon))\|_
{(W^{n}_p)^{m}\times\mathbb{C}^{rm}}\\
&\leq\gamma'\,\|y(\cdot,0)-y(\cdot,\varepsilon)\|_{n+r,p}
\end{align*}
for every $\varepsilon\in[0,\varepsilon')$. Thus, we obtain the left-hand side of the two-sided estimate \eqref{6.bound} with $\gamma_{1}:=1/\gamma'$.

Let us prove the right-hand side of this estimate. The boundary-value problem \eqref{bound_pr_eps_1}, \eqref{bound_pr_eps_2} satisfies Definition~\ref{defin_2} by Theorem~\ref{main_th}. Therefore the operator \eqref{6.LBe} is invertible for every $\varepsilon\in[0,\varepsilon_1)$, and, furthermore, its inverse $(L(\varepsilon),B(\varepsilon))^{-1}$ converges strongly to $(L(0),B(0))^{-1}$ as $\varepsilon\to0+$. Indeed, for arbitrary $f\in(W_p^{n})^{m}$ and $c\in\mathbb{C}^{rm}$, it follows from condition~$(\ast\ast)$ of Definition~\ref{defin_2} that
\begin{equation*}
(L(\varepsilon),B(\varepsilon))^{-1}(f,c)=:y(\cdot,\varepsilon)\to
y(\cdot,0):=(L(0),B(0))^{-1}(f,c)
\end{equation*}
in $(W_p^{n+r})^{m}$ as $\varepsilon\to0+$. Hence, there exist positive numbers $\varepsilon_{2}<\min\{\varepsilon_1,\varepsilon'\}$ and $\gamma_{2}$ such that
\begin{equation}\label{norm-inverse}
\|(L(\varepsilon),B(\varepsilon))^{-1}\|\leq\gamma_{2}\quad
\mbox{for every}\quad\varepsilon\in[0,\varepsilon_{2}).
\end{equation}
Here, of course, $\|(L(\varepsilon),B(\varepsilon))^{-1}\|$ denotes the norm of the inverse operator to \eqref{6.LBe}. Property \eqref{norm-inverse} is deduced from the above strong convergence and Banach-Steinhaus Theorem in the same way as that in the previous paragraph. Now, owing to this property, we conclude that
\begin{align*}
\|y(\cdot,0)-y(\cdot,\varepsilon)\|_{n+r,p}&=
\|(L(\varepsilon),B(\varepsilon))^{-1}(L(\varepsilon),B(\varepsilon))
(y(\cdot,0)-y(\cdot,\varepsilon))\|_{n+r,p}\\
&\leq\gamma_{2}\,\|(L(\varepsilon),B(\varepsilon))
(y(\cdot,0)-y(\cdot,\varepsilon))\|_{(W^{n}_p)^{m}\times\mathbb{C}^{rm}}\\
&=\gamma_{2}\,d_{n,p}(\varepsilon)
\end{align*}
for every $\varepsilon\in[0,\varepsilon_{2})$. Thus, we get the right-hand side of the estimate \eqref{6.bound}.
\end{proof}

\section{Application to multipoint boundary-value problems}\label{sec4}

We arbitrarily choose $\varkappa\geq1$ distinct points $t_{1},\ldots,t_{\varkappa}\in[a,b]$ and consider the following multipoint boundary-value problem:
\begin{gather}\label{bound_pr_1}
Ly(t)\equiv y^{(r)}(t)+ \sum\limits_{j=1}^{r}A_{r-j}(t)y^{(r-j)}(t)=
f(t),\quad a\leq t\leq b,\\
By\equiv\sum_{l=0}^{n+r-1}\sum_{i=1}^{\varkappa}
\alpha_{i}^{(l)}y^{(l)}(t_i)=c. \label{bagat_cond}
\end{gather}
Here, the unknown vector-valued function $y$ belongs to  $(W^{n+r}_p)^m$, whereas each matrix-valued function $A_{r-j}\in(W_p^n)^{m\times m}$, the vector-valued function $f\in(W^n_p)^m$, each number matrix $\alpha_{i}^{(l)}\in\mathbb{C}^{rm\times m}$, and the vector $c\in\mathbb{C}^{rm}$ are arbitrarily chosen.

Owing to the continuous embedding
\begin{equation}\label{W-C-embedding}
(W^{n+r}_{p})^m\hookrightarrow(C^{n+r-1})^m,
\end{equation}
the boundary condition \eqref{bagat_cond} is well posed, and the mapping $y\to By$, with $y\in(W^{n+r}_p)^m$, sets the continuous linear operator $B:(W^{n+r}_p)^m\to \mathbb{C}^{rm}$. Thus, the boundary-value problem \eqref{bound_pr_1}, \eqref{bagat_cond} is generic with respect to the Sobolev space $W^{n+r}_p$. Note that the boundary condition \eqref{bagat_cond} is not classical because it contains the derivatives  $y^{(l)}$ of order $l\geq r$ if $n\geq1$.

We consider \eqref{bound_pr_1}, \eqref{bagat_cond} as a limiting problem as $\varepsilon\to0+$ for the following multipoint boundary-value problem depending on the parameter $\varepsilon\in(0,\varepsilon_0)$:
\begin{gather}\label{bound_pr_2}
L(\varepsilon)y(t,\varepsilon)\equiv y^{(r)}(t,\varepsilon)+ \sum\limits_{j=1}^{r}A_{r-j}(t,\varepsilon)y^{(r-j)}(t,\varepsilon)=
f(t,\varepsilon),\quad a\leq t\leq b,\\
B(\varepsilon)y(\cdot,\varepsilon)\equiv\sum_{l=0}^{n+r-1} \sum_{i=0}^{\varkappa} \sum_{j=1}^{k_i} \alpha_{i,j}^{(l)}(\varepsilon) y^{(l)}(t_{i,j}(\varepsilon),\varepsilon)=c(\varepsilon). \label{bagat_cond_2}
\end{gather}
Here, for every $\varepsilon\in(0,\varepsilon_0)$, the unknown vector-valued function $y(\cdot,\varepsilon)$ belongs to  $(W^{n+r}_p)^m$, whereas each matrix-valued function $A_{r-j}(\cdot,\varepsilon)\in(W_p^n)^{m\times m}$, the vector-valued function $f(\cdot,\varepsilon)\in(W^n_p)^m$, each number matrix $\alpha_{i,j}^{(l)}(\varepsilon)\in\mathbb{C}^{rm\times m}$, every point $t_{i,j}(\varepsilon)\in[a,b]$, and the vector $c(\varepsilon)\in\mathbb{C}^{rm}$ are arbitrarily chosen. The positive integers $k_0,k_1,\ldots,k_\varkappa$ do not depend on $\varepsilon$.

Note that, unlike \eqref{bagat_cond}, the boundary condition \eqref{bagat_cond_2} is posed for the points $t_{i,j}(\varepsilon)$ united in $\varkappa+1$ sets $\{t_{i,1}(\varepsilon),\ldots,t_{i,k_i}(\varepsilon)\}$, with $i=0,1,\ldots,\varkappa$. This is caused by our further assumption on the behaviour of these points as $\varepsilon\to0+$. Namely, we will assume that $t_{i,j}(\varepsilon)\to t_{i}$ whenever $i\geq1$, whereas no assumption on convergence of the points $t_{0,j}(\varepsilon)$ will be made. Note also that the coefficients $A_{r-j}(t,\varepsilon)$ and  $\alpha_{i,j}^{(l)}(\varepsilon)$ and the points $t_{i,j}(\varepsilon)$ are not supposed to have any regularity with respect to $\varepsilon$.

Like \eqref{bound_pr_1}, \eqref{bagat_cond}, the boundary-value problem \eqref{bound_pr_2}, \eqref{bagat_cond_2} is generic with respect to $W^{n+r}_p$ for every $\varepsilon\in(0,\varepsilon_0)$. Thus, the matter of Section~\ref{sec2} is applicable to this problem provided that  we put $L(0):=L$ and $B(0):=B$. The main result of this section, Theorem~\ref{main_th}, gives necessary and sufficient conditions for the solution $y(\cdot,\varepsilon)$ to depend continuously on the parameter $\varepsilon$ at $\varepsilon=0$. Among them is Limit Condition~(II), which means the strong convergence $B(\varepsilon)\to B$ as $\varepsilon\to0+$ of continuous operators from $(W^{n+r}_p)^m$ to $\mathbb{C}^{rm}$. We will give explicit sufficient conditions under which this convergence holds true.

Note that it is scarcely possible to use the system of conditions (2a)--(2c) from Remark~\ref{rem2.4} for the verification of this strong convergence because it is difficult in practice to find the matrix-function $\Phi(\cdot,\varepsilon)$ in the canonical  representation~\eqref{oper_B_represent} of the operator $B(\varepsilon)$ corresponding to the multipoint boundary condition~\eqref{bagat_cond_2}.

\begin{theorem}\label{th4.1}
Suppose that the left-hand side of \eqref{bagat_cond_2} satisfies the following conditions as $\varepsilon \to 0+$:
\begin{itemize}
    \item[\textup{(d1)}] $t_{i,j}(\varepsilon)\to t_{i}$ for all $i\in\{1,\ldots,\varkappa\}$ and $j\in\{1,\ldots,k_i\}$; \smallskip
    \item[\textup{(d2)}] $\sum\limits_{j=1}^{k_i}
        \alpha_{i,j}^{(l)}(\varepsilon)\to\alpha_{i}^{(l)}$ for all $i\in\{1,\ldots,\varkappa\}$ and $l\in\{0,\ldots,n+r-1\}$; \smallskip
    \item[\textup{(d3)}] $\|\alpha_{i,j}^{(n+r-1)}(\varepsilon)\|\!\cdot\!
|t_{i,j}(\varepsilon)-t_{i}|^{1/q}=O(1)$
for all $i\in\{1,\ldots,\varkappa\}$ and $j\in\{1,\ldots,k_i\}$, \smallskip
    \item[\textup{(d4)}] $\|\alpha_{i,j}^{(l)}(\varepsilon)\|
        \!\cdot\!|t_{i,j}(\varepsilon)-t_{i}|\to0$ for all $i\in\{1,\ldots,\varkappa\}$, $j\in\{1,\ldots,k_i\}$, and $l\in\mathbb{Z}$ with $0\leq l\leq n+r-2$; \smallskip
    \item[\textup{(d5)}] $\alpha_{0,j}^{(l)}(\varepsilon)\to0$ for all $j\in\{1,\ldots,k_0\}$ and $l\in\{0,\ldots,n+r-1\}$.
\end{itemize}
\noindent Then \eqref{bagat_cond_2} satisfies Limit Condition~\textup{(II)}.
\end{theorem}

The hypothesis of this theorem need some comments. In conditions (d3) and (d4) we let $\|\cdot\|$ denote the norm of a number matrix, this norm being equal to the sum of the absolute values of all entries of the matrix. In condition~(d3), the number $q$ is defined by the formula $1/p+1/q=1$. If $p=1$, then $1/q=0$ and condition~(d3) means that $\|\alpha_{i,j}^{(n+r-1)}(\varepsilon)\|=O(1)$ as $\varepsilon\to0+$.  Conditions (d2) and (d4) admit that the coefficients $\alpha_{i,j}^{(l)}(\varepsilon)$ with $l\leq n+r-2$ may grow infinitely as $\varepsilon\to0+$ but not very rapidly. The same is true for the leading coefficients $\alpha_{i,j}^{(n+r-1)}(\varepsilon)$ in the $p>1$ case due to condition~(d3). Condition (d5) suggests that we need not assume any convergence of the points $t_{0,j}(\varepsilon)$ as $\varepsilon\to0+$, in contrast to condition~(d1).

The following result is a direct consequence of Theorems \ref{main_th} and \ref{th4.1}.

\begin{theorem}\label{th4.2}
Suppose that the multipoint boundary-value problem \eqref{bound_pr_2}, \eqref{bagat_cond_2} satisfies Limit Condition \textup{(I)} and conditions \textup{(d1)}--\textup{(d5)} and that the limiting problem \eqref{bound_pr_1}, \eqref{bagat_cond} with $f(\cdot)\equiv0$ and $c=0$ has only the trivial solution. Then the solution to the problem \eqref{bound_pr_2}, \eqref{bagat_cond_2} depends continuously on the parameter $\varepsilon$ at $\varepsilon=0$.
\end{theorem}

Note that the system of conditions (d1)--(d5) does not imply the uniform convergence $B(\varepsilon)\to B(0)$ as $\varepsilon\to0+$ of continuous operators from $(W^{n+r}_p)^m$ to $\mathbb{C}^{rm}$. Therefore the conclusion of Theorem~\ref{th4.2} does not follow from the Banach theorem on inverse operator.

\begin{proof}[Proof of Theorem $\ref{th4.1}$]
In view of Banach--Steinhaus Theorem, it is sufficient to prove that the norm of the operator $B(\varepsilon):(W^{n+r}_p)^m\rightarrow\mathbb{C}^{rm}$ is bounded as $\varepsilon\to0+$ and that $B(\varepsilon)y\to B(0)y$ in $\mathbb{C}^{rm}$ as $\varepsilon\to0+$ for every vector-valued function $y$ from the dense set
$$
(C^{\infty})^{m}:=C^{\infty}([a,b],\mathbb{C}^{m})
$$
in the space $(W^{n+r}_p)^m$.

Let us first prove the boundedness of the norm of $B(\varepsilon)$ as $\varepsilon\to0+$. We arbitrarily choose a vector-valued function $y\in(W^{n+r}_p)^m$ and a sufficiently small number $\varepsilon>0$. Owing to \eqref{bagat_cond} and \eqref{bagat_cond_2}, we have the inequality
\begin{equation}\label{6.mp.eq1}
\begin{aligned}
\|By-B(\varepsilon)y\|\leq &\sum\limits_{l=0}^{n+r-1} \sum\limits_{j=1}^{k_0}
\|\alpha_{0,j}^{(l)}(\varepsilon)\|\!\cdot\!
\|y^{(l)}(t_{0,j}(\varepsilon))\|\\
&\;\;\,+\sum\limits_{l=0}^{n+r-1}\sum\limits_{i=1}^{\varkappa}\,
\biggl\|\alpha_{i}^{(l)}y^{(l)}(t_{i})-
\sum\limits_{j=1}^{k_i}\alpha_{i,j}^{(l)}(\varepsilon)
y^{(l)}(t_{i,j}(\varepsilon))\biggr\|.
\end{aligned}
\end{equation}
Here, making use of the continuous embedding \eqref{W-C-embedding}, we can write
\begin{equation}\label{6.mp.eq22}
\begin{gathered}
\|\alpha_{0,j}^{(l)}(\varepsilon)\|\!\cdot\!
\|y^{(l)}(t_{0,j}(\varepsilon))\|\leq
c_0\|\alpha_{0,j}^{(l)}(\varepsilon)\|\!\cdot\!\|y\|_{n+r,p}
\end{gathered}
\end{equation}
for all $l\in\{0,\ldots,n+r-1\}$ and $j\in\{1,\ldots,k_0\}$, with $c_0$ denoting the norm of the embedding operator \eqref{W-C-embedding}.

Besides, given $l\in\{0,\ldots,n+r-1\}$ and $i\in\{1,\ldots,\varkappa\}$, we obtain the inequalities
\begin{equation}\label{6.mp.eq23}
\begin{aligned}
&\biggl\|\alpha_{i}^{(l)}y^{(l)}(t_{i})-
\sum\limits_{j=1}^{k_i}\alpha_{i,j}^{(l)}(\varepsilon)
y^{(l)}(t_{i,j}(\varepsilon))\biggr\|\\
&\leq\biggl\|\biggl(\alpha_{i}^{(l)}-\sum\limits_{j=1}^{k_i}
\alpha_{i,j}^{(l)}(\varepsilon)\biggl)\,y^{(l)}(t_{i})\biggr\|+
\biggl\|\,\sum\limits_{j=1}^{k_i}\alpha_{i,j}^{(l)}(\varepsilon)
\bigl(y^{(l)}(t_{i})-y^{(l)}(t_{i,j}(\varepsilon))\bigr)\biggr\|\\
&\leq c_0\biggl\|\alpha_{i}^{(l)}-\sum\limits_{j=1}^{k_i}
\alpha_{i,j}^{(l)}(\varepsilon)\biggr\|\!\cdot\!\|y\|_{n+r,p}+
\sum\limits_{j=1}^{k_i}
\|\alpha_{i,j}^{(l)}(\varepsilon)\|\!\cdot\!
\|y^{(l)}(t_{i,j}(\varepsilon))-y^{(l)}(t_{i})\|.
\end{aligned}
\end{equation}
Here, for $l=n+r-1$ and each $j\in\{1,\ldots,k_i\}$, we have the inequality
\begin{equation}\label{6.mp.eq2}
\begin{aligned}
&\|\alpha_{i,j}^{(n+r-1)}(\varepsilon)\|\!\cdot\!
\|y^{(n+r-1)}(t_{i,j}(\varepsilon))-y^{(n+r-1)}(t_{i})\|\\
&\leq
\|\alpha_{i,j}^{(n+r-1)}(\varepsilon)\|\,c_{1}\,\|y\|_{n+r,p}\,
|t_{i,j}(\varepsilon)-t_{i}|^{1/q},
\end{aligned}
\end{equation}
with $c_{1}$ being the norm of the continuous operator of the embedding of the Sobolev space $W^{n+r}_p$ in the complex H\"older space $C^{n+r-1,1/q}([a,b])$; see, e.g., \cite[Theorem~4.6.1(e)]{Triebel95}.
If $1/q=0$, then the latter space becomes $C^{n+r-1}$ and inequality \eqref{6.mp.eq2} holds true with $c_1:=2c_0$, of course. Besides, for each $l\in\mathbb{Z}$ with $0\leq l\leq n+r-2$, we conclude by the Lagrange theorem of the mean that
\begin{equation}\label{bound4.10}
\begin{aligned}
\|\alpha_{i,j}^{(l)}(\varepsilon)\|\!\cdot\!
\|y^{(l)}(t_{i,j}(\varepsilon))-y^{(l)}(t_{i})\|&\leq
\|\alpha_{i,j}^{(l)}(\varepsilon)\|\max_{a\leq t\leq b}\|y^{(l+1)}(t)\|
\!\cdot\!|t_{i,j}(\varepsilon)-t_{i}|\\
&\leq\|\alpha_{i,j}^{(l)}(\varepsilon)\|\,c_0\|y\|_{n+r,p}
|t_{i,j}(\varepsilon)-t_{i}|.
\end{aligned}
\end {equation}

Now it follows directly from the inequalities \eqref{6.mp.eq1}--\eqref{bound4.10} and conditions (d2)--(d5) that
$$
\|By-B(\varepsilon)y\|\leq c\,\|y\|_{n+r,p}
$$
for some number $c>0$ that does not depend on $y\in(W^{n+r}_p)^m$ and  sufficiently small $\varepsilon>0$. Hence, the norm of the operator $B(\varepsilon)$ is bounded as  $\varepsilon\to0+$.

Besides,
\begin{equation}\label{f4.11}
\|\alpha_{0,j}^{(l)}(\varepsilon)\|\!\cdot\!
\|y^{(l)}(t_{0,j}(\varepsilon))\|\to0\quad\mbox{as}\quad\varepsilon\to0+
\end{equation}
due to inequality \eqref{6.mp.eq22} and condition~(d5), and
\begin{equation}\label{f4.12}
c_0\biggl\|\alpha_{i}^{(l)}-\sum\limits_{j=1}^{k_i}
\alpha_{i,j}^{(l)}(\varepsilon)\biggr\|\!\cdot\!\|y\|_{n+r,p}\to0 \quad\mbox{as}\quad\varepsilon\to0+
\end{equation}
due to condition~(d2). Note that if $y\in(C^{\infty})^{m}$, then
\begin{equation}\label{f4.13}
\begin{aligned}
&\|\alpha_{i,j}^{(l)}(\varepsilon)\|\!\cdot\!
\|y^{(l)}(t_{i,j}(\varepsilon))-y^{(l)}(t_{i})\|\\
&\leq
\|\alpha_{i,j}^{(l)}(\varepsilon)\|\max_{a\leq t\leq b}\|y^{(l+1)}(t)\|
\!\cdot\!|t_{i,j}(\varepsilon)-t_{i}|\to0 \quad\mbox{as}\quad\varepsilon\to0+
\end{aligned}
\end{equation}
for each $l\in\{0,\ldots,n+r-1\}$ due to condition (d4) in the $0\leq l\leq n+r-2$ case and due to conditions (d1) and (d3) in the $l=n+r-1$ case. Now formulas \eqref{6.mp.eq1}, \eqref{6.mp.eq23}, and \eqref{f4.11}--\eqref{f4.13} yield the convergence $B(\varepsilon)y\to By$ in $\mathbb{C}^{rm}$ as $\varepsilon\to0+$ for every $y\in(C^{\infty})^{m}$.
\end{proof}


\begin{thebibliography}{99}

\bibitem{Ashordia1996}
M. Ashordia; \emph{Criteria of correctness of linear boundary value problems for systems of generalized ordinary differential equations}, Czechoslovak Math. J. 46 (1996), no.~3, 385--404.

\bibitem{Gikhman1952}
I. I. Gikhman; \emph{Concerning a theorem of N.~N.~Bogolyubov} (Russian), Ukr. Mat. Zh. 4 (1952), no.~2, 215--219.

\bibitem{GnypKodlyukMikhailets2015UMJ}
E. V. Gnyp (Ye. V. Hnyp), T. I. Kodlyuk, V. A. Mikhailets; \emph{Fredholm boundary-value problems with parameter in Sobolev spaces}, Ukrainian Math.~J. 67 (2015), no.~5, 658--667.

\bibitem{GoriunovMikhailets2010MN}
A. S. Goriunov, V. A. Mikhailets; \emph{Resolvent convergence of Sturm--Liouville operators with singular potentials}, Math. Notes 87 (2010), no. 1--2, 287--292.

\bibitem{GoriunovMikhailetsPankrashkin2013EJDE}
A. S. Goriunov, V. A. Mikhailets, K. Pankrashkin; \emph{Formally self-ajoint quasi-differential operators and boundary-value problems}, Electron. J. Differential Equations 2013 (2013), no.~101, 1--16.

\bibitem{GoriunovMikhailets2012UMJ}
A. S. Goriunov, V. A. Mikhailets; \emph{Regularization of two-term differential equations with singular coefficients by quasiderivatives}, Ukrainian Math. J. 63 (2012), no.~9, 1361--1378.

\bibitem{IoffeTixonov1974}
A. D. Ioffe, V. M. Tihomirov; \emph{Theory of Extremal Problems}, Stud. Math. Appl., Vol.~6, Noth-Holland Publishing Co., Amsterdam--New York, 1979.

\bibitem{KantorovichAkilov64}
L. V. Kantorovich, G. P. Akilov; \emph{Functional Analysis} [2nd edn], Pergamon Press, Oxford-Elmsford, N.Y., 1982.

\bibitem{Kiguradze1975}
I. T. Kiguradze; \emph{Some Singular Boundary-Value Problems for Ordinary Differential Equations} (Russian), Tbilisi University, Tbilisi, (1975).

\bibitem{Kiguradze1987}
I. T. Kiguradze; \emph{Boundary-value problems for systems of ordinary differential equations}, J.~Soviet Math. 43 (1988), no.~2, 2259--2339.

\bibitem{Kiguradze2003}
I. T. Kiguradze; \emph{On boundary value problems for linear differential systems with singularities}, Differ. Equ. 39 (2003), no.~2, 212--225.

\bibitem{Kodlyuk2012Dop11}
T. I. Kodlyuk; \emph{Multipoint boundary-value problems with parameter in Sobolev spaces} (Russian), Dopov. Nats. Akad. Nauk Ukr. Mat. Prirodozn. Tekh. Nauki (2012), no.~11, 15--19.

\bibitem{KodlyukMikhailets2013JMS}
T. I. Kodlyuk, V. A. Mikhailets; \emph{Solutions of one-dimensional boundary-value problems with parameter in Sobolev spaces}, J.~Math. Sci. (New York) 190 (2013), no.~4, 589--599.

\bibitem{KodlyukMikhailetsReva2013}
T. I. Kodlyuk, V. A. Mikhailets, N. V. Reva; \emph{Limit theorems for one-dimensional boundary-value problems}, Ukrainian Math. J. 65 (2013), no.~1, 77--90.

\bibitem{KrasnoselskiiKrein1955}
M. A. Krasnosel’skii, S. G. Krein; \emph{On the principle of averaging in nonlinear mechanics} (Russian), Uspekhi Mat. Nauk 10 (1955), no.~3, 147--153.

\bibitem{KurzweilVorel1957}
J. Kurzweil, Z. Vorel; \emph{Continuous dependence of solutions of differential equations on a parameter} (Russian), Czechoslovak Math. J. 7 (1957), no.~4, 568--583.

\bibitem{Levin1967}
A. Yu. Levin; \emph{The limiting transition for nonsingular systems $\dot{X}=A_n(t)X$} (Russian), Dokl. Akad. Nauk SSSR 176 (1967), no.~4, 774--777.

\bibitem{MikhailetsChekhanova2014DAN7}
V. A. Mikhailets, G. A. Chekhanova; \emph{Fredholm boundary-value problems with parameter on the spaces $C^{(n)}[a;b]$} (Russian), Dopov. Nats. Akad. Nauk Ukr. Mat. Prirodozn. Tekh. Nauki (2014), no.~7, 24--28.

\bibitem{MikhailetsChekhanova2015}
V. A. Mikhailets, G. A. Chekhanova; \emph{Limit theorems for general one-dimensional boundary-value problems}, J. Math. Sci. (New York) 204 (2015), no.~3, 333--342.

\bibitem{MikhailetsMurachSoldatov2016EJQTDE}
V. A. Mikhailets, A. A. Murach, V. Soldatov; \emph{Continuity in a parameter of solutions to generic boundary-value problems}, Electron. J. Qual. Theory Differ. Equ. 2016 (2016), no.~87, 1--16.

\bibitem{MikhailetsMurachSoldatov2016MFAT}
V. A. Mikhailets, A. A. Murach, V. Soldatov; \emph{A criterion for continuity in a parameter of solutions to generic boundary-value problems for higher-order differential systems}, Methods Funct. Anal. Topology 22 (2016), no.~4, 375--386.


\bibitem{Nguen1993}
T. K. Nguen; \emph{On the dependence of a solution to a linear system of differential equations on a parameter}, Differ. Equa. 29 (1993), no.~6, 830--835.

\bibitem{Opial1967}
Z. Opial; \emph{Continuous parameter dependence in linear systems of differential equations}, J.~Differential Equations, 3 (1967), no.~4, 571--579.

\bibitem{ReedSimon80}
M. Reed, B. Simon; \emph{Methods of Modern Mathematical Physics. I: Functional Analysis}, Academic Press, London, 1980.

\bibitem{Reid1967}
W. T. Reid; \emph{Some limit theorems for ordinary differential systems}, J.~Differential Equations~3 (1967), no.~3, 423--439.

\bibitem{RieszSz-Nagy56}
F. Riesz, B. Sz-Nagy; \emph{Functional Analysis}, Blackie \& Son Limited, London and Glasgow, 1956.

\bibitem {Soldatov2015UMJ}
V. O. Soldatov; \emph{On the continuity in a parameter for the solutions of boundary-value problems total with respect to the spaces  $C^{(n+r)}[a,b]$}, Ukrainian Math.~J. 67 (2015), no.~5, 785--794.

\bibitem{Triebel95}
H. Triebel; \emph{Interpolation Theory, Function Spaces, Differential Operators} [2-nd edn], Heidelberg, Johann Ambrosius Barth, 1995.

\end{thebibliography}
\end{document}